\documentclass[12pt]{amsart}

\usepackage{amsmath, amsthm, amssymb, enumerate, amscd}
\usepackage[mathscr]{eucal}
\usepackage{pgf,tikz}
\usepackage{tikz, tikz-cd}
\usetikzlibrary{matrix,decorations.pathmorphing,arrows}

\theoremstyle{plain}
\newtheorem{theorem}{Theorem}[section]
\newtheorem{proposition}[theorem]{Proposition}
\newtheorem{corollary}[theorem]{Corollary}

\theoremstyle{definition}
\newtheorem{definition}[theorem]{Definition}

\newtheorem{example}[theorem]{Example}

\newcommand{\N}{{\mathbb{N}}}
\newcommand{\Z}{{\mathbb{Z}}}
\newcommand{\Q}{{\mathbb{Q}}}
\newcommand{\R}{{\mathbb{R}}}

\begin{document}

\title[AP monoid domains]{For which additive submonoids $M$ of nonnegative rationals is $F[X;M]$ AP?}

\author{Ryan Gipson}
\address{Department of Mathematics\\ 
University of Louisville\\
Louisville, KY 40292, USA}
\email{ryan.gipson@louisville.edu}
\thanks{$\ast$ the corresponding author}
\author{Hamid Kulosman$^\ast$}
\address{Department of Mathematics\\ 
University of Louisville\\
Louisville, KY 40292, USA}
\email{hamid.kulosman@louisville.edu}



\date{}
\begin{abstract} 
We characterize the submonoids $M$ of the additive monoid $\Q_+$ of nonnegative rational numbers for which the irreducible and the prime elements in the monoid domain $F[X;M]$ coincide. We present a diagram of implications between some types of submonoids of $\Q_+$,  with a precise position of the monoids $M$ with this property. 
\end{abstract}
\maketitle

\section{Introduction}\label{intro}
If $M$ is a commutative monoid, written additively, and $F$ is a field, the monoid ring $F[X;M]$ consists of the polynomial expressions (also called polynomials)
\[f=a_0X^{\alpha_0}+a_1X^{\alpha_1}+\dots +a_nX^{\alpha_n},\]
where $n\ge 0$, $a_i\in F$, and $\alpha_i\in M$ $(i=0,1,\dots, n)$. $R$ is an integral domain if and only if $M$ is cancellative and torsion-free. If $M=\N_0$, then $R=F[X]$, which is a PID and so, in particular, the notions of irreducible elements (also called atoms) and prime elements coincide in $R$. The domains having this property are called AP domains. However, if $M=\langle 2,3\rangle=\{0,2,3,4,\dots\}$ (the submonoid of $\N_0$ generated by the elements $2$ and $3$), then, for example, the elements $X^2$ and $X^3$ of $R$ are atoms that are not prime. If we do not restrict ourselves to submonoids of $\N_0$, but consider the submonoids of $\Q_+=\{q\in\Q\;:\; q\ge 0\}$ instead, then, for example, the monoid domain $R=F[X;\Q_+]$ is an AP domain (by a theorem of R. Daileda \cite{d}), while in the monoid domain $R=F[X;M]$ with $\displaystyle{M=\langle \frac{1}{2}, \frac{1}{2^2}, \frac{1}{2^3},\dots; \frac{1}{5}\rangle}$ the element $\displaystyle{X^{1/5}}$ is an atom which is not prime (see \cite{gk}). So we naturally come to the following {\bf question}: {\it for which submonoids $M$ of $\Q_+$ is the monoid domain $F[X;M]$ AP?} The goal of this paper is to give an answer to this question. In addition to that, we will present an implication diagram between various properties of submonoids of $\Q_+$ in which we will precisely position the monoids $M$ for which $F[X;M]$ is AP.

\section{Notation and preliminaries}
We begin by recalling some definitions and statements. All the notions that we use  but not define in this paper, as well as the definitions and statements for which we do not specify the source,  can be found in the classical reference books \cite{c_book} by P.~M.~Cohn, \cite{gil_mit} and \cite{gil} by R.~Gilmer, \cite{k} by I.~Kaplansky, and \cite{n} by D.~G.~Northcott, as well as in our papers \cite{cgk} and \cite{gk}. We also recommend the paper \cite{a_df} in which the work of R.~Gilmer is nicely presented, in particular his work on characterizing cancellative torsion-free monoids $M$ for which the monoid domain $F[X;M]$ has the property $P$ for various properties $P$.

\smallskip
We use $\subseteq$ to denote inclusion and $\subset$ to denote strict inclusion. We denote $\N=\{1,2,\dots\}$ and $\N_0=\{0,1,2,\dots\}$. An element $\displaystyle{\frac{m}{n}}$ of $\Q_+=\{q\in\Q\;:\; q\ge 0\}$  is said to be in {\it reduced form} if $\gcd(m,n)=1$. If $\displaystyle{\frac{m_1}{n_1}, \frac{m_2}{n_2}}$ are two elements of $\Q_+$ in reduced form, then $\displaystyle{\frac{m_1}{n_1}=\frac{m_2}{n_2}}$ if and only if $m_1=m_2$ and $n_1=n_2$ (see \cite{gk}).

All the monoids used in the paper are assumed to be commutative and written additively. Thus a {\it monoid} is a nonempty set $M$ with an associative and commutative operation $+:M\times M\to M$, possessing an {\it identity element} $0\in M$ (such that $0+x=x$ for all $x\in M$). We say that a monoid $M$ is {\it cancellative} if for any elements $x,y,z\in M$, $x+y=x+z$ implies $y=z$. A monoid $M$ is {\it torsion-free} if for any $n\in \N$ and $x,y\in M$, $nx=ny$ implies $x=y$. If $M,M'$ are two monoids, a map $\mu: M\to M'$ is a {\it monoid homomorphism} if $\mu(x+y)=\mu(x)+\mu(y)$ for all $x,y\in M$, and $\mu(0)=0$. If it is also bijective, then it is called a {\it monoid isomorphism}. Then $\mu^{-1}:M'\to M$ is also a monoid isomorphism. If $M,M'$ are submonoids of $\Q_+$, then every isomorphism of $M$ onto $M'$ has the form $\mu_\tau(x)=\tau x$, where $\tau$ is a positive rational number (as it is easy to see). Then $M'=\tau M=\{\tau x\;:\; x\in M\}$. For a monoid $M\subseteq \Q_+$, the {\it difference group} of $M$ is the subgroup $\mathrm{Diff}(M)=\{x-y\;:\;x,y\in M\}$ of $\Q$.

A subset $I$ of a monoid $M$ is called is called an {\it ideal} of $M$ if $M+I=I$, i.e., if for every $a\in I$, $M+a\subseteq I$. (Here $S_1+S_2=\{x+y\;:\; x\in S_1,\, y\in S_2\}$ for any two subsets $S_1, S_2$ of $M$.) An ideal $I$ of $M$ is said to be {\it principal} if there is an element $a\in I$ such that $I=M+a$. We then write $I=(a)$. A submonoid of a monoid $M$ generated by a subset $A\subseteq M$ is denoted by $\langle A\rangle$ (while an ideal of $M$, generated by a subset $A\subseteq M$, is denoted by $(A)$, in order to avoid eventual confusions). 
We have: $\langle A\rangle=\{n_1a_1+\dots+n_ta_t\::\; t\ge 0, n_i\in\N_0, a_i\in A\; (i=1,2,\dots, t)\}$. We assume $\langle\emptyset\rangle=\{0\}$. A monoid $M$ is said to be {\it cyclic} if it can be generated by one element.
We introduced in \cite{gk} the notion of an {\it essential generator} of a monoid $M$, that is an element $a\in M$ such that $\langle M\setminus \{a\}\rangle \ne M$. Note that the essential generators of $M$ are precisely the {\it atoms} of $(M,+)$, i.e., the non-invertible elements $a\in M$ such that $a=b+c$ $(b,c\in M)$ implies that at least one of the elements $b,c$ is invertible in $M$. We say that a monoid $M$ is {\it atomic} if it can be generated by essential generators. (We also assume that the monoid $M=\{0\}$ is atomic since it can be generated by the empty set of essential generators, i.e., atoms.) We say that a monoid $M$ is an {\it ACCP monoid} if every  if every increasing sequence 
\[(a_1) \subseteq (a_2) \subseteq (a_3)\subseteq \dots\]
of principal ideals of $M$ is stationary, meaning that $(a_{n_0})=(a_{n_0+1})=(a_{n_0+2})=\dots$ for some $n_0$. (Note that the notions of an atomic and an ACCP monoid are analogous to the notions of an atomic and an ACCP integral domain. The similarities and differences between the ideal theories of monoids and integral domains are studied, for example, in \cite{k1}.)

\begin{proposition}[{\cite{gk}}]\label{prop_1}
Let $M$ be a monoid and let $A$ be its generating set. 

(a) If $a$ is an essential generator of $M$, then $a\in A$.

(b) If $a\in A$ is such that $\langle A\setminus \{a\}\rangle\ne M$, then $a$ is an essential generator of $M$.

(c) If $a_1, \dots, a_n\in A$ are nonessential generators of $M$, then $A'=A\setminus \{a_1,\dots, a_n\}$ is also a generating set of $M$. 
\end{proposition}

\begin{proposition}[{\cite{gk}}]\label{prop_2}
Let $\mu:M\to M'$ be a monoid isomorphism and $a\in M$. Then $a$ is an essential generator of $M$ if and only if $\mu(a)$ is an essential generator of $M'$.
\end{proposition}

The notion of a  Pr\"ufer monoid was introduced in \cite[p. 223-224]{gp} (see also \cite[p. 166-167]{gil}). We include the possibility that $M=\{0\}$.

\begin{definition}
We say that a monoid $M$ is:

(a) a {\it  Pr\"ufer monoid} it it is a union of an increasing sequence of cyclic submonoids;

(b) {\it difference-closed} if for any $a,b\in M$ with $a\ge b$ we have $a-b\in M$.
\end{definition}

\begin{definition}
We say that a monoid $M\subseteq \Q_+$ is: 

(a) a {\it half-group monoid} if there is a subgroup $G$ of the additive group $\Q$ of rational numbers such that $M=G\cap \Q_+$;

(b) {\it integrally closed} if for every $n\in\N$ and any $x,y\in M$ with $x\ge y$, $n(x-y)\in M$ implies $x-y\in M$.
\end{definition}

 In this paper all rings are {\it integral domains}, i.e., commutative rings with identity in which $xy=0$ implies $x=0$ or $y=0$.
A non-zero non-unit element $x$ of an integral domain $R$ is said to be {\it irreducible} (and called an {\it atom}) if $x=yz$ with $y,z\in R$ implies that $y$ or $z$ is a unit. A non-zero non-unit element $x$ of an integral doman $R$ is said to be {\it prime} if $x\mid yz$ with $y,z\in R$ implies $x\mid y$ or $x\mid z$. Every prime element is an atom, but not necessarily vice-versa. Two elements $x,y\in R$ are said to be {\it associates} if $x=uy$, where $u$ is a unit. We then write $x\sim y$. An element $x$ of $R$ is said to be {\it primal} if $x\mid ab$ for some $a,b\in R$ implies $x=x_1x_2$ for some $x_1, x_2\in R$ with $x_1\mid a$ and $x_2\mid b$. 

Let $R,T$ be two integral domains with $R\subseteq T$. An element $x\in T$ is {\it integral over $R$} if it is a root of a monic polynomial in the polynomial ring $R[X]$. The subring $R'$ of $T$ consisting of the elements of $T$ integral over $R$ is called the {\it integral closure of $R$ in $T$}. If $R'=R$, $R$ is said to be {\it integrally closed in $T$}. The integral closure of $R$ in its field of fractions is called the {\it integral closure} of $R$. 

\begin{definition}
An integral domain  $R$ is said to be:

(a) a {\it Euclidean domain} if there is a function $w:R\setminus \{0\}\to\N_0$ such that for any $a\in R$ and $b\in\R\setminus  \{0\}$ there exist $q,r\in R$ such that $a=bq+r$, and such that either $r=0$ or $w(r)<w(b)$.

(b) a {\it principal ideal domain} (PID) if every ideal of $R$ is principal. 

(c) a {\it unique factorization domain} (UFD) if it is atomic and for every non-zero, non-unit $x\in R$, every two factorizations of $x$ into atoms are equal up to order and associates. 

(d) {\it integrally closed} if it is equal to its integral closure.

(e) {\it Noetherian} if every ascending sequence of ideals of it stabilizes.

(f) an {\it ACCP domain} if every ascending sequence of principal ideals of it stabilizes.

(g) a {\it Dedekind domain} if it is a Noetherian integrally closed domain whose every prime ideal is maximal.

(h) a {\it B\'ezout domain} if every two-generated ideal of $R$ is principal. 

(i) a {\it GCD domain} if any two elements of it have a greatest common divisor.

(j) a {\it pre-Schreier domain} if every element $x\in R$ is {\it primal}.

(k) a {\it Schreier domain} if it is pre-Schreier and integrally closed.

(l) a {\it  Pr\"ufer domain} if each domain $S$ such that $R\subseteq S\subseteq K$, where $K$ is the field of fractions of $R$, is integrally closed.

(m) an {\it AP domain} if every atom of it is prime.

(n) {\it atomic} if every non-zero non-unit element of it can be written as a (finite) product of atoms. 

(o) a {\it MIP domain} if all maximal ideals of it are principal.

(p) a {\it PC domain} if every proper two-generated ideal of it is contained in a proper principal ideal.
\end{definition}

\begin{proposition}\label{implications_domains}
Diagram 1 (see below) of implications between some types of integral domains holds.
\end{proposition}

\begin{proof}
Most of these implications are well-known or/and follow from the definitions. 
The implications B\'ezout $\Rightarrow$ GCD $\Rightarrow$ Schreier $\Rightarrow$ pre-Schreier $\Rightarrow$ AP can be found in \cite{c}. That Dedekind $\Rightarrow$ Pr\"ufer can be seen in \cite[page 169]{gil}.
An example that a pre-Schreier domain does not have to be integrally closed is given in \cite{z}. The domain $R=k[X, XY, XY^2,\dots]$, where $k$ is a field (known as {\it G.~Evans' example}),  is an example of an integrally closed domain which is not AP. Indeed, by \cite[page 51]{h}, or \cite[page 114]{k}, $R$ is integrally closed and ACCP. It is not AP, otherwise it would be a UFD, but it is not, as $X\cdot XY^2=XY\cdot XY$ are two decompositions of $X^2Y^2$ into non-associated atoms. The PC domains were introduced in our paper \cite{cgk_PID}, where all the implications involving this type of domains can be found.
\end{proof}

Later, when we consider all these types of domains in the context of the monoid domains $F[X;M]$ with $M\subseteq \Q_+$, the types of domains inside inside each of the two circles on the diagram will coincide. That was already known for the small circle and for all the types from the big circle except the AP domains. The fact that the AP type of domains coincide with other types from the big circle in the context  of the monoid domains $F[X;M]$ with $M\subseteq \Q_+$ is the main result of this paper. 

\section{Some properties of monoids $M\subseteq \Q_+$ and the associated monoid domains $F[X;M]$}

The next proposition is a characterization of when $F[X;M]$ is a Euclidean domain, or a PID, or a Dedekind domain.

\begin{proposition}[{\cite[Theorem 8.4]{gp_mich}}]\label{ED_PID_Dedekind}
Let $M$ be a submonoid of $\Q_+$. The following are equivalent:

(a) $M=\{0\}$ or $M\cong \N_0$;

(b) $F[X;M]$ is a Euclidean domain;

(c) $F[X;M]$ is a PID;

(d) $F[X;M]$ is a Dedekind domain.
\end{proposition}

\phantom{X}
\vspace{.5cm}
\hspace{-3cm}
\begin{tikzpicture}
\node[draw] (NOETH) at (1.01,-8) {Noetherian};
\node[draw] (DED) at (1,-11) {Dedekind};
\node[draw] (A) at (2.99,-3) {Atomic};
\node[draw] (ACCP) at (3,-6) {ACCP};
\node[draw] (PID) at (3,-13) {PID};
\node[draw] (EUCL) at (2.995,-14.5) {Euclidean};
\node[draw] (FIELD) at (3, -16.5) {Field};
\node[draw] (ID) at (8.3,0) {Integral domain};
\node[draw] (UFD) at (5,-11) {UFD};
\node[draw] (INT_CL) at (8.31,-5) {\begin{tabular}{c} Integrally\\ closed\end{tabular}};
\node[draw] (PRUF) at (8.3, -8.2) {Pr\"ufer};
\node[draw] (AP) at (11.985,-3) {AP};
\node[draw] (PRE_SCH) at (11.98,-5) {Pre-Schreier};
\node[draw] (SCH) at (11.99,-7) {Schreier};
\node[draw] (GCD) at (12,-8.6) {GCD};
\node[draw] (BEZ) at (11.99,-10.6) {B\'ezout};
\node[draw] (PC) at (14.2,-8.6) {PC};
\node[draw] (MIP) at (15,-10.5) {MIP};
\draw[-implies, double equal sign distance] (A) -- (ID);
\draw[-implies, double equal sign distance] (INT_CL) -- (ID);
\draw[-implies, double equal sign distance] (AP) -- (ID);
\draw[-implies, double equal sign distance] (ACCP) -- (A);
\draw[-implies, double equal sign distance] (NOETH) -- (ACCP);
\draw[-implies, double equal sign distance] (DED) --  (NOETH);
\draw[-implies, double equal sign distance] (DED) -- (PRUF);
\draw[-implies, double equal sign distance] (UFD) -- (ACCP);
\draw[-implies, double equal sign distance] (PID) -- (DED);
\draw[-implies, double equal sign distance] (PID) -- (UFD);
\draw[-implies, double equal sign distance] (EUCL) -- (PID);
\draw[-implies, double equal sign distance] (FIELD) -- (EUCL);
\draw[-implies, double equal sign distance] (PRUF) -- (INT_CL);
\draw[-implies, double equal sign distance] (UFD) -- (GCD);
\draw[-implies, double equal sign distance] (BEZ) -- (PRUF);
\draw[-implies, double equal sign distance] (BEZ) -- (GCD);
\draw[-implies, double equal sign distance] (BEZ) -- (PC);
\draw[-implies, double equal sign distance] (GCD) -- (SCH);
\draw[-implies, double equal sign distance] (SCH) -- (PRE_SCH);
\draw[-implies, double equal sign distance] (SCH) -- (INT_CL);
\draw[-implies, double equal sign distance] (PRE_SCH) -- (AP);
\draw[-implies, double equal sign distance] (MIP) -- (PC);
\draw[-implies, double equal sign distance] (PC) ..controls (14,-4).. (AP);
\draw[-implies, double equal sign distance] (PID) -- (BEZ);
\draw[-implies, double equal sign distance] (PID) ..controls (12.5,-12).. (MIP);
\node[circle, draw] (first_circle) at (2.82,-12.3) {\phantom{XXXXXXXXXXXXXXXXXXXXX}};
\node[circle, draw] (second_circle) at (11.5,-7.3) {\phantom{XXXXXXXXXXXXXXXXXXXXXXXXXXXXXXXXXXXX}};
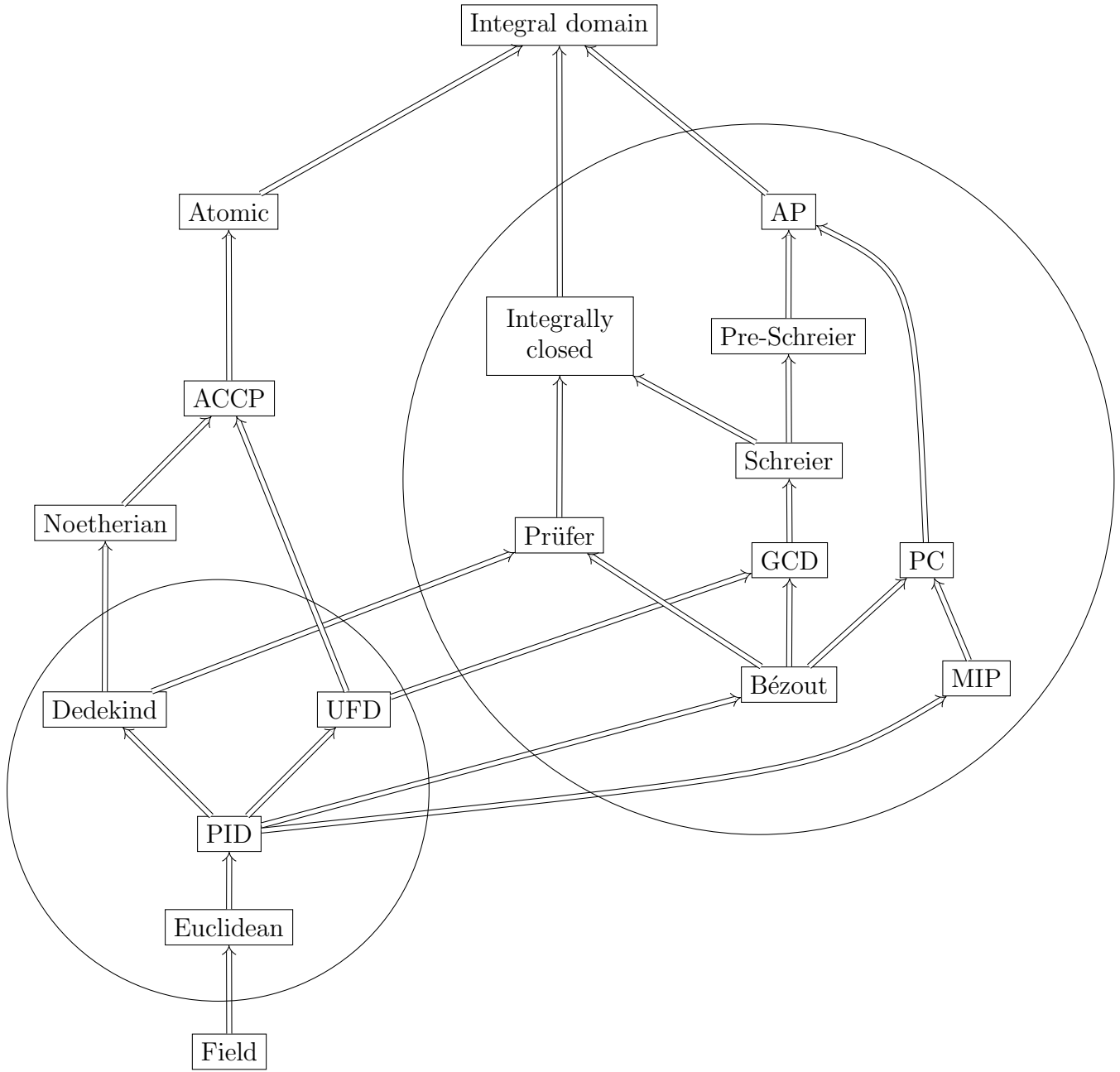
\node (title) at (8.5,-18) {Diagram 1: Implications between some types of integral domains};
\end{tikzpicture}

\newpage
The next proposition is a characterization of when $F[X;M]$ is a Pr\"ufer domain, or a B\'ezout domain, or an integrally closed domain.

\begin{proposition}[{\cite[Theorem (iv)]{gp}, \cite[Corollary 12.11]{gil}, \cite[Theorem 13.5]{gil}}]\label{Pr_Be_IntCl}
Let $M$ be a submonoid of $\Q_+$. The following are equivalent:

(a) $M$ is integrally closed;

(b) $F[X;M]$ is a  Pr\"ufer domain;

(c) $F[X;M]$ is a B\'ezout domain;

(d) $F[X;M]$ is integrally closed.
\end{proposition}

The atomicity of the monoid domains $F[X;M]$ for submonoids $M$ of $\Q_+$ is characterized by the next proposition.

\begin{proposition}[{\cite[Theorem 3.6]{g}}]\label{atomicity_thm}
If $M$ is a submonoid of $\Q_+$, then $F[X;M]$ is atomic if and only if $M$ can be generated by essential generators.
\end{proposition}

With respect to the AP-ness of the monoid domains $F[X;M]$ for $M\subseteq \Q_+$ we have the next statement from our paper \cite{gk}.

\begin{proposition}[{\cite[Propositions 2.8 and 2.9]{gk}}]\label{X_to_a}
Let $M$ be a submonoid of $\Q_+$, not isomorphic to $\N_0$. Then the irreducible elements of $F[X;M]$ of the form $X^a$, $a\in M$, are precisely the $X^a$ with $a$ an essential generator of $M$, and they are all non-prime.  
\end{proposition}

To further study the AP-ness of the monoid domains $F[X;M]$ for $M\subseteq \Q_+$ we introduced in \cite{gk} the following notion.

\begin{definition}[{\cite{gk}}]
We say that a monoid $M\subseteq \Q_+$ satisfies the {\it gcd/lcm condition} if for any $t\in \N$ and any elements $\displaystyle{\frac{m_1}{n_1}, \frac{m_2}{n_2},\dots, \frac{m_t}{n_t}}\in M$, written in reduced form, at least one of which is $\ne 0$,  we have $\displaystyle{\frac{\gcd(m_1,\dots, m_t)}{\mathrm{lcm}(n_1,\dots, n_t)}\in M}$.
\end{definition}

It is easy to see that an {\it equivalent definition} of the gcd/lcm condition is that for any {\it two} elements $\displaystyle{\frac{m_1}{n_1}, \frac{m_2}{n_2}\in M}$, written in reduced form, at least one of which is $\ne 0$,  we have $\displaystyle{\frac{\gcd(m_1, m_2)}{\mathrm{lcm}(n_1, n_2)}\in M}$.

\smallskip
The next statement is a generalization of the theorem of Daileda, mentioned in Introduction. (The proof that we gave in \cite{gk} follows Daileda's proof of his theorem from \cite{d}.) The statement is a step toward the main theorem of this paper, namely Theorem \ref{main_thm}, and is in fact used in its proof.

\begin{proposition}[{\cite[Theorem 5.4]{gk}}]\label{gcd_lcm_implies_AP}
If the monoid $M\subseteq \Q_+$ satisfies the gcd/lcm condition, then for any field $F$ the monoid domain $F[X;M]$ is AP.
\end{proposition}

For a prime number $p$ the notion of the {\it $p$-height} $h_p(a)$ of an element $a$ of a torsion-free group $G$ is defined in \cite[page 108 ]{f} as the nonnegative integer $r$ such that $a\in p^rG\setminus p^{r+1}G$ if such an integer exists and as $\infty$ otherwise. The sequence $(h_2(a), h_3(a), h_5(a), \dots)$ of $p$-heights of $a$ as $p$ goes through all prime numbers in the increasing order is called there the {\it height sequence} of $a$. In our paper \cite{cgk} we considered the elements of height $(0,0,0,\dots)$, but, more generally, in the torsion-free monoids instead of groups.

\begin{definition}[{\cite{cgk}}]
We say that an element $a$ of a torsion-free monoid $M$ is of height $(0,0,0,\dots)$ if for every prime number $p$ the equation $px=a$ cannot be solved for an $x\in M$.
\end{definition} 

\begin{example}\label{ht_0_0_0_not_ess_gen}
Note that every essential generator of $M$ is an element of height $(0,0,0,\dots)$. The converse does not hold. For example, in the monoid 
\[M=\langle\frac{1}{2}, \frac{1}{2^2}, \frac{1}{2^3}\dots;  \frac{1}{3}, \frac{1}{3^2}, \frac{1}{3^3}, \dots \rangle\subseteq \Q_+\] 
the element $\displaystyle{\frac{5}{6}}$ is of height $(0,0,0,\dots)$, but is not an essential generator.
\end{example}

\begin{proposition}[{\cite{cgk}}]\label{ht_0_elem_preserved_under_isoms}
Let $\mu:M\to M'$ be a monoid isomorphism. An element $a\in M$ is of height $(0,0,0,\dots)$ in $M$ if and only if $\mu(a)$ is of height $(0,0,0,\dots)$ in $M'$.
\end{proposition}

The next statement is the main theorem of our paper \cite{cgk}. It is not true for fields of positive characteristics (as the examples and related questions in \cite{cgk} illustrate). We use this statement in the proof of the main theorem of this paper (namely Theorem \ref{main_thm}) and that explains why we need in it the assumption that $F$ is of characteristic $0$.

\begin{theorem}[{\cite[Theorem 4.1]{cgk}}]\label{irr_thm}
Let $M$ be a submonoid of $\Q_+$, $F$ a field of characteristic $0$, and $\pi$ an element of $M$ of height $(0,0,0,\dots)$. Then the binomial $X^\pi-1$ is irreducible in $F[X;M]$.
\end{theorem}

\section{When is $F[X;M]$ an AP domain?}
The equivalence of the conditions (a), (c') and (f) from the next theorem was proved in \cite[Theorem 13.5]{gil}, we include the proofs for the sake of completeness.

\begin{theorem}\label{main_thm_prep}\label{conditions_thm}
Let $M$ be a submonoid of $\Q_+$ and $F$ a field. The following are equivalent:

(a) $M$ is a Pr\"ufer monoid;

(b) $M$ is difference-closed;

(c) $M$ is a half-group monoid;

(c') $M=\mathrm{Diff}(M)\cap \Q_+$;

(d) $M$ satisfies the gcd/lcm condition;

(e) $M\cong \N_0$ or $M$ has no elements of height $(0,0,0,\dots)$;

(f) $M$ is integrally closed.
\end{theorem}

\begin{proof}
(i) \, \underbar{(a) $\Rightarrow$ (b):} Suppose 
\[M=\langle\frac{m_1}{n_1}\rangle\cup \langle\frac{m_2}{n_2}\rangle \cup \langle\frac{m_3}{n_3}\rangle \cup \dots ,\]
where
\[\langle\frac{m_1}{n_1}\rangle \subseteq \langle\frac{m_2}{n_2}\rangle \subseteq \langle\frac{m_3}{n_3}\rangle \subseteq \dots\]
Let $\displaystyle{\frac{a}{b}, \frac{c}{d}\in M}$ with $\displaystyle{\frac{c}{d}>\frac{a}{b}}$. There is an $r\in\N_0$ such that $\displaystyle{\frac{a}{b}, \frac{c}{d}\in \langle\frac{m_r}{n_r}\rangle}$, 
hence 
\[\displaystyle{\frac{c}{d}=k\frac{m_r}{n_r}},\, \displaystyle{\frac{a}{b}=l\frac{m_r}{n_r}}, \text{ with $k>l$},\]
 and so 
 \[\displaystyle{\frac{c}{d}-\frac{a}{b}=(k-l)\frac{m_r}{n_r}\in  \langle\frac{m_r}{n_r}\rangle\subseteq M}.\]
Thus $M$ is difference closed.

\medskip\noindent
\underbar{(b) $\Rightarrow$ (c):} Let $M$ be difference closed and let $G=M\cup (-M)$, where $-M=\{-m\;:\;m\in M\}\subseteq \Q$. We need to show that $G$ is a group. We only need to show that the sum of two elements of $G$ is an element of $G$ (everything else being clear). Let $a,b\in G$. If $a,b\in M$, or $a,b\in (-M)$, clear. Suppose $a\in M$, $b\in (-M)$. Let $b=-c$, $c\in M$. If $a\ge c$, then $a-c\in M$, hence $a+b\in G$. If $a<c$, then $c-a\in M$, hence $a-c\in (-M)$, hence $a+b\in G$. Thus $M$ is a half-group monoid.

\medskip\noindent
\underbar{(c) $\Rightarrow$ (a):} Let $M=G\cap \Q_+$, where $G$ is a subgroup of $(\Q,+)$. If $G=\{0\}$, clear. Suppose $G\ne \{0\}$. The group $\Q$ is a union of an increasing sequence of cyclic subgroups (for example, $\displaystyle{\Q=\cup_{n=1}^\infty\langle\frac{1}{n!}\Z\rangle}$), i.e., 
\[\Q=\langle q_1\rangle\cup \langle q_2\rangle \cup \langle q_3\rangle \cup \dots,\]
where
\[\langle q_1\rangle\subseteq \langle q_2\rangle \subseteq \langle q_3\rangle \subseteq\dots\]
and
\[q_1>q_2>q_3>\dots >0.\]
 Let $n_{i}$ be the smallest natural number such that $n_{i}q_{i}\in G$ $(i=1,2,3,\dots)$. Then 
\[G=\langle n_{1}q_{1}\rangle\cup \langle n_{2}q_{2}\rangle \cup \langle n_{3}q_{3}\rangle\cup \dots,\]
where
\[\langle n_{1}q_{1}\rangle\subseteq \langle n_{2}q_{2}\rangle \subseteq \langle n_{3}q_{3}\rangle\subseteq \dots\]
Hence 
\[M=(\langle n_{1}q_{1}\rangle\cap\Q_+)\cup (\langle n_{2}q_{2}\rangle\cap\Q_+) \cup (\langle n_{3}q_{3}\rangle\cap\Q_+)\cup \dots,\]
where
\[(\langle n_{1}q_{1}\rangle\cap \Q_+)\subseteq (\langle n_{2}q_{2}\rangle\cap\Q_+) \subseteq (\langle n_{3}q_{3}\rangle\cap\Q_+)\subseteq \dots\]
and each of $\langle n_{i}q_{i}\rangle\cap \Q_+$ is a cyclic submonoid of $M$ (generated by $n_{i}q_{i}$).
Thus $M$ is a Pr\"ufer monoid.

\medskip\noindent
\underbar{(c) $\Rightarrow$ (c'):} If $M=G\cap \Q_+$ for some subgroup $G$ of $\Q$, then clearly $\mathrm{Diff}(M)\subseteq(G)$. On the other side, if $x\in G$, then either $x\in M$, in which case $x=x-0\in\mathrm{Diff}(M)$, or $-x\in M$, in which case $x=0-(-x)\in \mathrm{Diff}(M)$. Thus $G=\mathrm{Diff}(M)$, so that $M=\mathrm{Diff}(M)\cap \Q_+$.

\medskip\noindent
\underbar{(c') $\Rightarrow$ (c):} Clear.

\medskip\noindent
\underbar{(b) $\Rightarrow$ (d):} Let $\displaystyle{\frac{m_1}{n_1}, \frac{m_2}{n_2}}$ be two elements of $M$, written in reduced form, with $\displaystyle{\frac{m_1}{n_1}<\frac{m_2}{n_2}}$. Let $d=\gcd(m_1, m_2)$ and $e=\gcd(n_1, n_2)$, both positive. Then 
\begin{align*}
m_1=dx_1, &\quad  m_2=dx_2, & \!\!\!\!\!\! \gcd(x_1, x_2)=1, &\,  \\
n_1=ey_1, &\quad   n_2=ey_2, & \!\!\!\!\!\!\!\!\! \gcd(y_1, y_2)=1, &\;\;\; \mathrm{lcm}(n_1, n_2)=ey_1y_2.
\end{align*}
Since $\gcd(x_1y_2, x_2y_1)=1$, there are $k,l\in\N_0$ such that 
\begin{equation}\label{gcd_lk}
dkx_2y_1-dlx_1y_2=d \text{ or $-d$}.
\end{equation}
We will assume that this difference is equal to $d$, the reasonong being similar if it is equal to $-d$. Since
\[\frac{m_1}{n_1}=\frac{dx_1}{ey_1}\in M,\]
we have
\[\frac{dlx_1}{ey_1}\in M.\]
Also
\[\frac{m_2}{n_2}=\frac{dx_2}{ey_2}\in M,\]
hence
\[\frac{dkx_2}{ey_2}\in M.\]
Hence, since $M$ is difference closed,
\[\frac{dkx_2}{ey_2}-\frac{dlx_1}{ey_1}\in M,\]
i.e., 
\[\frac{dkx_2y_1-dlx_1y_2}{ey_1y_2}\in M.\]
Hence by (\ref{gcd_lk}),
\[\frac{d}{ey_1y_2}\in M,\]
i.e., 
\[\frac{\gcd(m_1, m_2)}{\mathrm{lcm}(n_1, n_2)}\in M.\]
Thus $M$ satisfies the gcd/lcm property.

\medskip\noindent
\underbar{(d) $\Rightarrow$ (b):}  Let $\displaystyle{\frac{m_1}{n_1}, \frac{m_2}{n_2}}$ be two elements of $M$, written in reduced form, with $\displaystyle{\frac{m_1}{n_1}<\frac{m_2}{n_2}}$. Let $d=\gcd(m_1, m_2)$ and $e=\gcd(n_1, n_2)$, both positive. Then 
\begin{align*}
m_1=dx_1, &\quad  m_2=dx_2, & \!\!\!\!\!\! \gcd(x_1, x_2)=1, &\,  \\
n_1=ey_1, &\quad   n_2=ey_2, & \!\!\!\!\!\!\!\!\! \gcd(y_1, y_2)=1, &\;\;\; \mathrm{lcm}(n_1, n_2)=ey_1y_2.
\end{align*}
We have
\[\frac{\gcd(m_1, m_2)}{\mathrm{lcm}(n_1, n_2)}=\frac{d}{ey_1y_2}\in M.\]
Hence $M$ contains the element
\[\frac{d}{ey_1y_2}=\frac{d(x_2y_1-x_1y_2)}{ey_1y_2}=\frac{dx_2}{ey_2}-\frac{dx_1}{ey_1}=\frac{m_2}{n_2}-\frac{m_1}
{n_1}.\]
Thus $M$ is difference closed.

\medskip\noindent
\underbar{(e) $\Rightarrow$ (c):} Assume that $M$ has no elements of height $(0,0,0,\dots)$ Let
\[q_1, q_2, q_3,\dots\]
be the list of all the elements of $M$. Since $M$ has no elements of height $(0,0,0,\dots)$, there are elements $\displaystyle{\frac{q_1}{p_1}, \frac{q_1}{p_1p_2}, \frac{q_1}{p_1p_2p_3},\dots}$ ($p_i$ prime numbers) in $M$, so that the union $M_1$ of cyclic submonoids
\[\langle q_1 \rangle\subseteq \langle\frac{q_1}{p_1}\rangle\subseteq\langle  \frac{q_1}{p_1p_2}\rangle\subseteq\langle  \frac{q_1}{p_1p_2p_3}\rangle\subseteq\dots\]
is a (Pr\"ufer, hence) half-group submonoid of $M$ containing $q_1$. Similarly, there is a half-group submonoid $M_2$ of $M$ containing $q_2$. Since the sum of two half-group submonoids of $\Q_+$ is a half-group submonoid of $\Q_+$, the monoid $M_{1,2}=M_1+M_2$ is a half-group submonoid of $M$ containing $q_1$ and $q_2$, and $M_1\subseteq M_{1,2}$. Continuing in a similar way we construct an increasing sequence 
\[M_1\subseteq M_{1,2}\subseteq M_{1,2,3}\subseteq\dots \]
of half-group submonoids of $M$, such that $M_{1,2,\dots, n}$ contains $q_1, q_2, \dots, q_n$ $(n=1,2,3,\dots)$ Then the union 
\[M=M_1\cup M_{1,2}\cup M_{1,2,3}\cup\dots\]
is a half-group submonoid of $\Q_+$.

\medskip\noindent
\underbar{(d) $\Rightarrow$ (e):} Assume that $M$ satisfies the gcd/lcm property and is not isomorphic to $\N_0$. Suppose to the contrary, i.e., that $M$ has an element of height $(0,0,0,\dots)$, say $\displaystyle{\frac{m_1}{n_1}}$, written in reduced form. Then $\displaystyle{\frac{m_1}{n_1}}\ne 0$. Let $\displaystyle{\frac{m_2}{n_2}}$ be another element of $M$, written in reduced form, which is not in $\displaystyle\langle{\frac{m_1}{n_1}\rangle}$. Let  $d=\gcd(m_1, m_2)$ and $e=\gcd(n_1, n_2)$, both positive. Then 
\begin{align*}
m_1=dx_1, &\quad  m_2=dx_2, & \!\!\!\!\!\! \gcd(x_1, x_2)=1, &\,  \\
n_1=ey_1, &\quad   n_2=ey_2, & \!\!\!\!\!\!\!\!\! \gcd(y_1, y_2)=1, &\;\;\; \mathrm{lcm}(n_1, n_2)=ey_1y_2.
\end{align*}
Hence 
\begin{equation*}
\frac{\gcd(m_1, m_2)}{\mathrm{lcm}(n_1, n_2)}=\frac{d}{ey_1y_2}\in M.
\end{equation*}
We can now write
\begin{equation}\label{gcd_lcm_eq}
\frac{m_1}{n_1}=\frac{dx_1}{ey_1}=x_1y_2\cdot \frac{d}{ey_1y_2}.
\end{equation}
If $x_1=y_2=1$, then $m_1=d$, so that $m_1\mid n_1$, and $n_2=e$, so that $n_1\mid n_2$. Hence 
\[\frac{m_1}{n_1}=\frac{d}{ey_1},\;\;\;\; \frac{m_2}{n_2}=\frac{dx_2}{e},\]
so that 
\[\frac{m_2}{n_2}=x_2y_1\, \frac{m_1}{n_1}\in\langle \frac{m_1}{n_1}\rangle,\]
a contradiction. Hence either $x_1\ne 1$, or $y_2\ne 1$. But then (\ref{gcd_lcm_eq}) implies that $\displaystyle{\frac{m_1}{n_1}}$ is not of height $(0,0,0,\dots)$, a contradiction.

\medskip\noindent
\underbar{(f) $\Rightarrow$ (b):} Clear.

\medskip\noindent
\underbar{(b) $\Rightarrow$ (f):} If $\displaystyle{\frac{m_1}{n_1}, \frac{m_2}{n_2}\in M}$ with $\displaystyle{\frac{m_1}{n_1}\ge \frac{m_2}{n_2}}$, then $\displaystyle{m_1n_1n_2\cdot (\frac{m_1}{n_1}- \frac{m_2}{n_2})}$ $=m_1(m_1n_2-m_2n_1)\in M$ (as $m_1\in M$), hence  $\displaystyle{\frac{m_1}{n_1}- \frac{m_2}{n_2}\in M}$. Thus $M$ is difference closed.
\end{proof}

\begin{corollary}\label{isoms_preserve_conditions}
Let $M, M'$ be two isomorphic submonoids of $\Q_+$. Then $M$ satisfies any of the (equivalent) conditions from the previous theorem if and only if $M'$ satisfies it.
\end{corollary}

\begin{proof}
It is enough to prove this for one of the conditions (a)-(e). Since all the isomorphisms $\mu:M\to M'$ have the form $\mu_\tau: x\to \tau x$ for every $x\in M$ with $\tau$ being any positive rational number, it is enough to see that $M$ is difference-closed if and only if $\tau M$ is difference closed, which is clear.  Hence the statement follows. 
\end{proof}

We can now prove the main theorem of the paper. It characterizes the submonoids $M$ of $\Q_+$ for which $F[X;M]$ is an AP domain when $F$ is a field of characteristic $0$.  
 
\begin{theorem}\label{main_thm} 
Let $M$ be a submonoid of $\Q_+$ and $F$ a field.

(i) Each of the conditions (a), (b), (c), (c'), (d), (e), (f) from Theorem \ref{main_thm_prep} implies that $F[X;M]$ is an AP domain. 
 
(ii) If $F$ is of characteristic $0$, then each of the conditions from Theorem \ref{main_thm_prep} is equivalent with $F[X;M]$ being an AP domain. 
\end{theorem}

\begin{proof}
(i) It is enough to show that one of the (equivalent) conditions (a)-(e) implies the condition $F[X;M]$ AP. Hence the claim (i) follows Theorem 5.4 from our paper \cite{gk} (see the above Proposition \ref{gcd_lcm_implies_AP}) which states that if $M$ satisfies the gcd/lcm condition, then $F[X;M]$ is AP.

\smallskip\noindent
(ii) Because of (i) it is enough to show that the condition $F[X;M]$ AP implies one of the (equivalent) conditions (a)-(e). We will show that it implies the condition (e). Let $F$ be a field of characteristic $0$. Assume $F[X;M]$ is AP and $M$ is not isomorphic to $\N_0$. If $M=\{0\}$, clear. So we also assume that $M\ne \{0\}$. By Proposition \ref{X_to_a}, $M$ is infinitely generated and without essential generators. Suppose to the contrary of the statement, i.e., that $M$ has an element of height $(0,0,0,\dots)$, say $\pi$. Since by the hypothesis $\pi$ is not an essential gene\-rator, we have $\pi=\pi_1+\pi_2$ for some non-zero elements $\displaystyle{\pi_1=\frac{m_1}{n_1}}$ and $\displaystyle{\pi_2=\frac{m_2}{n_2}}$, both written in reduced form. We will show that the element $X^\pi-1$ of $F[X;M]$ is irreducible, but not prime. Under the monoid isomorphism $\tau_{n_1n_2}:M\to M'=n_1n_2M$ the element $\pi$ is mapped to the element $\pi'=n_1n_2\pi=m_1n_2+n_1m_2$. The monoid $M'$ is infinitely generated, without essential generators, and $\pi'$ is an element of $M'$ of height $(0,0,0,\dots)$ which is a sum of two non-zero elements of $M'$, namely $m_1n_2$ and $n_1m_2$. Moreover, $X^\pi-1$ is irreducible non-prime in $F[X;M]$ if and only if $X^{\pi'}-1$  is irreducible non-prime in $F[X;M']$. 

So it is enough to prove that, if the monoid $M$ in our statement contains an element $\pi$ of height $(0,0,0,\dots)$ such that $\pi=m_1+m_2$, where $m_1, m_2$ are relatively prime elements of $\N$, then $X^\pi-1$ is irreducible but not prime. Note that $m_1\ne 1$ and $m_2\ne 1$, otherwise $\pi$ wouldn't be of height $(0,0,0,\dots)$. So $m_1$ has at least one prime factor, say $p$.

Under the monoid isomorphism $\displaystyle{\tau_{1/m_2p}:M\to M'=\frac{1}{m_2p}\,M}$ the element $\pi=m_1+m_2$ is mapped to the element 
\[\pi'=\frac{1}{m_2p}\,\pi=\frac{m_1'}{m_2} +\frac{1}{p},\]
where $\displaystyle{m_1'=\frac{m_1}{p}}$. The monoid $M'$ is infinitely generated, without essential generators, and $\pi'$ is an element of $M'$ of height $(0,0,0,\dots)$ which is a sum of two nonzero elements of $M'$, namely $\displaystyle{\frac{m_1'}{m_2}}$ and $\displaystyle{\frac{1}{p}}$, both written in reduced form. Moreover, $X^\pi-1$ is irreducible non-prime in $F[X;M]$ if and only if $X^{\pi'}-1$ is irreducible non-prime in $F[X;M']$.

So it is enough to prove that, if the monoid $M$ in our statement contains an element $\pi$ of height $(0,0,0,\dots)$ such that $\displaystyle{\pi=\frac{m}{n}+\frac{1}{p}}$, where $\gcd(m,n)=1$, $\gcd(n,p)=1$, $n\ne 1$, then $X^\pi-1$ is irreducible non-prime.

So let us assume that we have this situation. Then, by Theorem \ref{irr_thm}, $X^\pi-1$ is irreducible in $F[X;M]$. We will show that it is not prime. We have $\displaystyle{\pi=\frac{mp+n}{pn}}$. Hence 
\[X^\pi-1=(X^{\frac{mp+n}{pn}}-1)\mid (X^{\frac{mp+n}{p}}-1).\]
We have:
\begin{align*}
X^{\frac{mp+n}{p}}-1 &= (X^{\frac{1}{p}})^{mp+n}-1\\
                                      &= (X^{\frac{1}{p}}-1)\,(X^{\frac{mp+n-1}{p}}+X^{\frac{mp+n-2}{p}}+\dots+X^{\frac{1}{p}}+1).
\end{align*}
Since $\displaystyle{\frac{1}{p}<p}$, we have $\displaystyle{(X^\pi-1)\nmid (X^{\frac{1}{p}}-1)}$. Suppose that 
\begin{equation}\label{eq_0}
(X^\pi-1)\mid (X^{\frac{mp+n-1}{p}}+X^{\frac{mp+n-2}{p}}+\dots+X^{\frac{1}{p}}+1).
\end{equation}
Then
\begin{align}\label{eq_star}
X^{\frac{mp+n-1}{p}} &+ X^{\frac{mp+n-2}{p}}+\dots+X^{\frac{1}{p}}+1\notag\\
                                      &= (X^{\frac{mp+n}{pn}}-1)\,(X^{\alpha_1}+g_2X^{\alpha_2}+\dots+g_{k-1}X^{\alpha_{k-1}}-1),
\end{align}
where $\alpha_1>\alpha_2>\dots >\alpha_{k-1}>0$ and $g_2, \dots, g_{k-1}\in F$. It follows that
\[\alpha_1=\frac{(mp+n)(n-1)-n}{pn}.\]
Note that 
\[\displaystyle{\alpha_1\ne \frac{mp+n-i}{p}} \text{ for all $i=2,3,\dots, mp+n-1$},\] since, otherwise, we would get $(i-2)n=mp$, which is not possible since $\gcd(m,n)=1$ and $\gcd(p,n)=1$. Note also that
\[\alpha_1<\frac{mp+n-2}{p}.\]
Hence the exponent $\displaystyle{\frac{mp+n-2}{p}}$ on the left-hand side of (\ref{eq_star}) has to be obtained from 
\[X^{\frac{mp+n}{pn}}\cdot (X^{\alpha_1}+g_2X^{\alpha_2}+\dots+g_{k-1}X^{\alpha_{k-1}}-1),\]
since it cannot be obtained from
\[(-1)\cdot  (X^{\alpha_1}+g_2X^{\alpha_2}+\dots+g_{k-1}X^{\alpha_{k-1}}-1).\]
(These two are parts of the left-hand side.) But then we must have either
\begin{equation}\label{eq_1}
\frac{mp+n}{pn}+\alpha_2=\frac{mp+n-2}{p},
\end{equation}
or, if 
\begin{equation}\label{eq_1_prim}
\frac{mp+n}{pn}+\alpha_i=\frac{mp+n-2}{p}\; \text{ for some $i\ge 3$},
\end{equation}
then the terms with the exponents $\displaystyle{\frac{mp+n}{pn}+\alpha_2, \dots, \frac{mp+n}{pn}+\alpha_{i-1}}$ would have to be cancelled on the right-hand side, so that we would have
\begin{align*}
\frac{mp+n}{pn}+\alpha_2 &= \alpha_1,\\
\frac{mp+n}{pn}+\alpha_3 &= \alpha_2,\\
\hspace{1cm} \dots\dots\dots\dots & \dots\dots\\
\frac{mp+n}{pn}+\alpha_{i-1} &= \alpha_{i-2}.
\end{align*}
Hence
\begin{align*}
\alpha_2 &= \frac{(mp+n)(n-2)-n}{pn},\\
\alpha_3 &= \frac{(mp+n)(n-3)-n}{pn},\\
.\dots & \dots\dots\dots\dots\dots\dots\dots\dots.\\\
\alpha_{i-1} &= \frac{(mp+n)(n-i+1)-n}{pn},\\
\end{align*}
and from (\ref{eq_1_prim})
\[\alpha_i = \frac{(mp+n)(n-i)-2n}{pn}.\]
Now since $\alpha_{i-1}>\alpha_i$, we would have
\[\frac{(mp+n)(n-i+1)-n}{pn}>\frac{(mp+n)(n-1)-2n}{pn},\]
which gives 
\[n>(mp+n)(i-2) \; \text{ for $i\ge 3$},\]
which is not true. Hence (\ref{eq_1}) holds. This gives
\begin{equation}\label{eq)2}
\alpha_2=\frac{(mp+n)(n-1)-2n}{pn}.
\end{equation}
Note that 
\[\displaystyle{\alpha_2\ne \frac{mp+n-i}{p}} \text{ for all $i=3,4,\dots, mp+n-1$},\] since, otherwise, we would get $(i-3)n=mp$, which is not possible since $\gcd(m,n)=1$ and $\gcd(p,n)=1$. Note also that
\[\alpha_2<\frac{mp+n-3}{p}.\]
Hence the exponent $\displaystyle{\frac{mp+n-3}{p}}$ on the left-hand side of (\ref{eq_star})  has to be obtained from 
\[X^{\frac{mp+n}{pn}}\cdot (X^{\alpha_1}+g_2X^{\alpha_2}+\dots+g_{k-1}X^{\alpha_{k-1}}-1)\]
(we noticed before that it cannot be equal to $\alpha_1$).

Reasoning in the same way as before we conclude that
\begin{equation}\label{eq_3}
\alpha_3=\frac{(mp+n)(n-1)-3n}{pn}.
\end{equation}
By induction we get
\begin{equation}\label{eq_4}
\alpha_i=\frac{(mp+n)(n-1)-in}{pn}
\end{equation}
for $i=1,2,3,\dots, r$, where $r$ is the largest integer for which 
\begin{equation}\label{eq_5}
\frac{mp+n}{pn}<\frac{mp+n-r}{p}.
\end{equation}
From (\ref{eq_5}) we get 
\[r<mp+n-1-\frac{mp}{n}.\]
Hence
\[r\le mp+n-2.\]
However, then 
\begin{equation}\label{eq_6}
\frac{mp+n-r}{p}\ge \frac{2}{p},
\end{equation}
so that not all of the exponents $\displaystyle{\frac{mp+n-i}{p}}$ from the left-hand side are obtained as $\displaystyle{\frac{mp+n}{pn}+\alpha_i}$.  As before,
$\displaystyle{\alpha_r=\frac{(mp+n)(n-1)-rn}{pn}}$ cannot be equal to any $\displaystyle{\frac{mp+n-i}{p}}$, $i=r+1, r+2, \dots, mp+n-1$, otherwise we would get $n(i-r-1)=mp$, which is not possible since $\gcd(m,n)=1$ and $\gcd(p,n)=1$. Now note that
\begin{equation}\label{eq_7}
\alpha_r<\frac{mp+n-(r+1)}{p}.
\end{equation}
It follows that the exponent $\displaystyle{\frac{mp+n-(r+1)}{p}}$ from the left-hand side cannot be obtained from 
\[X^{\frac{mp+n}{pn}}\cdot (X^{\alpha_1}+g_2X^{\alpha_2}+\dots+g_{k-1}X^{\alpha_{k-1}}-1)\]
since $r$ was the largest integer for which (\ref{eq_5}) holds, nor from
\[(-1) (X^{\alpha_1}+g_2X^{\alpha_2}+\dots+g_{k-1}X^{\alpha_{k-1}}-1)\]
since (\ref{eq_7}) holds and no $\alpha_i$ with $i<r$ can be equal to any $\displaystyle{\frac{mp+n-i}{p}}$, $i=r+1, r+2, \dots, mp+n-1$. We got a contradiction, hence (\ref{eq_star}) holds, i.e., (\ref{eq_0}) does not hold. So $X^\pi-1$ is not prime.

The theorem is proved.
\end{proof}


\vspace{.15cm}
\hspace{-2.8cm}
\begin{tikzpicture}
\node[draw] (ALL_M) at (7,0) {M submonoid of $\Q_+$};
\node[draw] (ONE_ESS) at (2,-2.5) {\begin{tabular}{c} $M$ has at least one\\ essential generator\\ or $M=\{0\}$\end{tabular}};
\node[draw] (GEN_BY_ESS) at (1.99,-5.5) {\begin{tabular}{c} $M$ can be generated\\ by essential generators,\\ i.e., $M$ is atomic\end{tabular}};
\node[draw] (ACCP_COND) at (2.01,-8.4) {\begin{tabular}{c} $M$ is ACCP\\ $\equiv$ $F[X;M]$  ACCP\end{tabular}};
\node[draw] (FIN_GEN) at (2, -10.8) {\begin{tabular}{c} $M$ finitely generated\\ $\equiv$ $F[X;M]$ Noetherian\end{tabular}};
\node[draw] (ZERO_OR_N_ZERO) at (2.01, -14) 
      {\begin{tabular}{c} 
             $M=\{0\}$ or $M\cong \N_0$\\ 
             $\equiv$ $F[X;M]$ Dedekind domain\\
             $\equiv$ $F[X;M]$ UFD\\
             $\equiv$ $F[X;M]$ PID\\
             $\equiv$ $F[X;M]$ Euclidean domain\\ 
      \end{tabular}
      };
\node[draw] (ZERO) at (0.7, -17) {\begin{tabular}{c} $M=\{0\}$ \\ $\equiv$ $F[X;M]$ field \end{tabular}};
\node[draw] (N_ZERO) at (4.2,-17) {$M\cong\N_0$};
\node[draw] (NOESS_OR_N_ZERO) at (13.01,-2.5) {\begin{tabular}{c} $M$ has no essential\\ generators or $M \cong \N_0$ \end{tabular}};
\node[draw] (AP) at (13,-8) 
      {\begin{tabular}{c} 
           $M$ satisfies the conditions\\ 
           from Thm \ref{conditions_thm}\\
          $\equiv$ $F[X;M]$ pre-Schreier;\\
          $\equiv$ $F[X;M]$ Schreier;\\
          $\equiv$ $F[X;M]$ GCD;\\
          $\equiv$ $F[X;M]$ Pr\"ufer\\ 
          $\equiv$ $F[X;M]$ B\'ezout\\
          $\equiv$ $F[X;M]$ integ closed\\
          $\equiv$ $F[X;M]$ MIP\\
          $\equiv$ $F[X;M]$ PC\\
          \vspace{.4cm}
          \phantom{X}
      \end{tabular}
      };
\node[draw, rectangle, rounded corners=3pt, fill=gray!20] (AP_text) at (13,-10.46) {$\equiv$ $F[X;M]$ AP}; 
\draw[-implies, double equal sign distance] (ONE_ESS) -- (ALL_M);
\draw[-implies, double equal sign distance] (GEN_BY_ESS) -- (ONE_ESS);
\draw[-implies, double equal sign distance] (ACCP_COND) -- (GEN_BY_ESS);
\draw[-implies, double equal sign distance] (FIN_GEN) -- (ACCP_COND);
\draw[-implies, double equal sign distance] (ZERO_OR_N_ZERO) -- (FIN_GEN);
\draw[-implies, double equal sign distance] (ZERO) --  (ZERO_OR_N_ZERO);
\draw[-implies, double equal sign distance] (N_ZERO) -- (ZERO_OR_N_ZERO);
\draw[-implies, double equal sign distance] (NOESS_OR_N_ZERO) -- (ALL_M);
\draw[-implies, double equal sign distance] (AP) -- (NOESS_OR_N_ZERO);
\draw[-implies, double equal sign distance] (ZERO_OR_N_ZERO) ..controls (12.5, -13.4).. (AP);
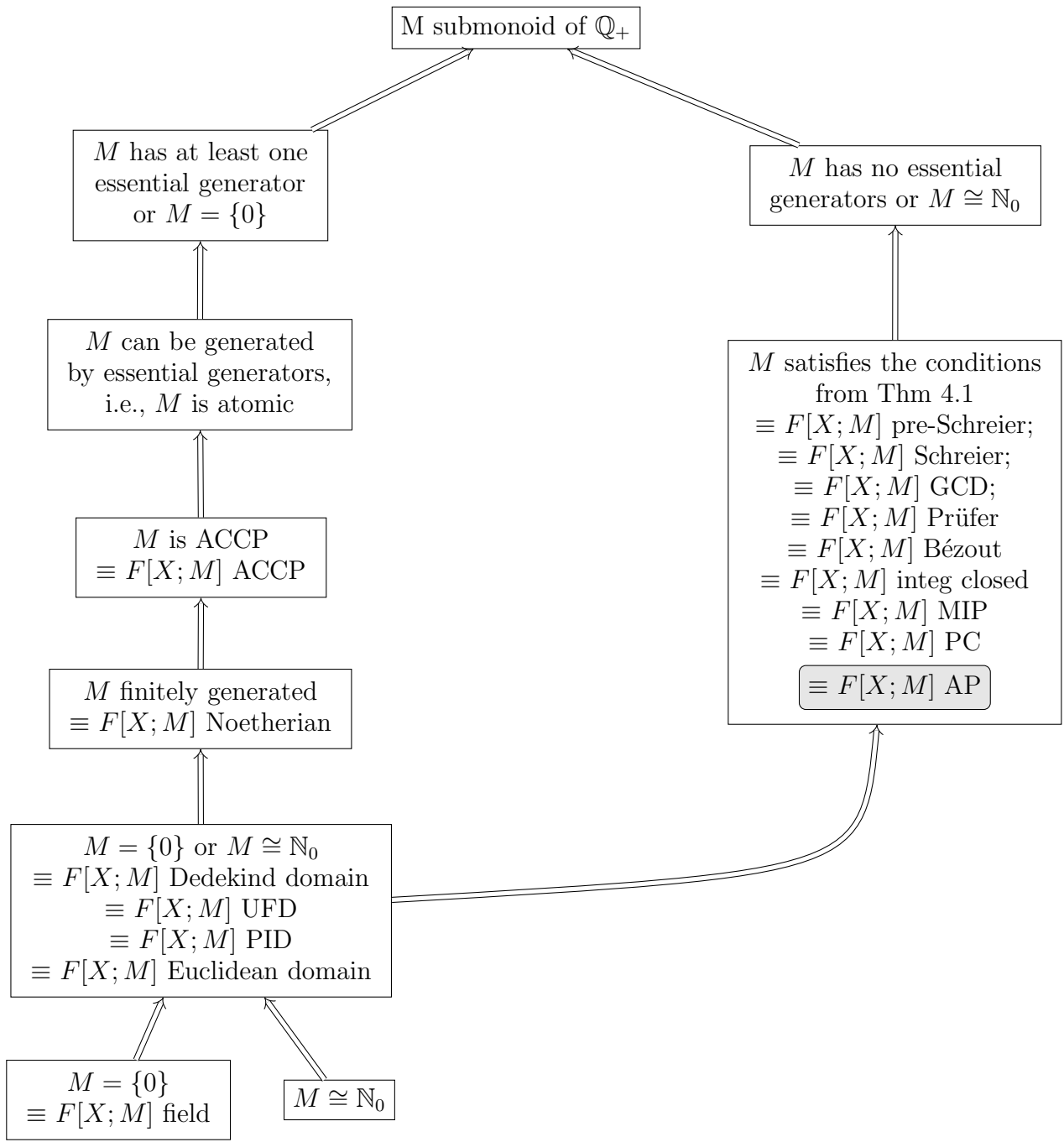
\node (title) at (7.5,-19) {Diagram 2: Implications between some types of submonoids of $\Q_+$};
\end{tikzpicture}

\newpage
\section{Implication diagram}
\begin{proposition}\label{implications_monoids}
Diagram 2 (see above) of implications between some types of submonoids of $\Q_+$ holds.
\end{proposition}

\begin{proof}
We have: 

(1) All the types in the second rectangle from the top in the second column of the diagram are equivalent by Proposition \ref{implications_domains}, Proposition \ref{Pr_Be_IntCl}, Theorem \ref{conditions_thm}, and Theorem \ref{main_thm}. 

(2) All the types, but UFD, in the second rectangle from the bottom in the first column of the diagram are equivalent by Proposition \ref{implications_domains} and Proposition \ref{ED_PID_Dedekind}. 
Since, by (1), UFD $\Leftrightarrow$ atomic + AP $\Leftrightarrow$ atomic + B\'ezout $\Leftrightarrow$ PID, we have that UFD as well is equivalent with these types.

(3) The equivalence in the third level from the bottom rectangle in the first column of the diagram follows from \cite[Theorem 7.7]{gil}.

(4) The equivalence in the third rectangle from the top in the first column of the diagram follows from \cite{k2}. 

(5) The implication and its strictness from the second rectangle to the first rectangle in the second column of the diagram follows from Example \ref{ht_0_0_0_not_ess_gen}. 

(6) The implication and its strictness from the third  rectangle from the top to the second rectangle from the top  in the first column of the diagram follows from \cite{k2}.

(7) All the remaining implications and their strictness are clear.
\end{proof}

\section{Concluding remarks}
(1) In our paper \cite{gk} we asked in Question 5.7 if for any monoid $M\subseteq \Q_+$ without essential generators $F[X;M]$ is AP. We can now answer this question in the negative. Namely, the monoid 
\[M=\langle\frac{1}{2}, \frac{1}{2^2}, \frac{1}{2^3}\dots;  \frac{1}{3}, \frac{1}{3^2}, \frac{1}{3^3}, \dots \rangle\subseteq \Q_+\] 
that we considered in Example \ref{ht_0_0_0_not_ess_gen} does not have essential generators, however by Theorem \ref{main_thm} $F[X;M]$ is not AP for any field $F$ of characteristic $0$ since $M$ contains elements of height $(0,0,0,\dots)$, 
the element $\displaystyle{\frac{5}{6}}$, for example, being one of them. 

\medskip
(2) In Question 5.7 from our paper \cite{gk} we also asked for a characterization of all monoids $M\subseteq \Q_+$ such that $F[X;M]$ is AP. The main theorem of this paper (namely, Theorem \ref{main_thm}) answers this question for fileds $F$ of characteristic $0$. In the proof of our main theorem we use Theorem \ref{irr_thm} from our paper \cite{cgk}, which is not true for fields of positive characteristic. So we may naturally ask the following {\bf question}:  {\it for a given characteristic $p>0$ characterize all the submonoids $M$ of $\Q_+$ such that $F[X;M]$ is AP for any field $F$ of characteristic $p$. In parti\-cu\-lar, characterize all the submonoids $M$ of $\Q_+$ such that $F[X;M]$ is AP for any field $F$ of positive characteristic.} 

\medskip
(3) Another question that we had in \cite{gk} was Question 2.11 in which we asked if for any monoid $M\subseteq\Q_+$ which can be generated by essential generators, $F[X;M]$ is atomic. (In other words, if $M$ atomic is equivalent with $F[X;M]$ atomic, as the other direction was proved in our \cite[Proposition 2.10]{gk}.) This question is a special case of the question of Gilmer (\cite[page 189]{gil}) about determining the conditions under which the domain $D[X;M]$ is atomic, where $D$ is a domain and $M$ is a cancellative torsion-free monoid.  For the case of torsion-free abelian groups instead of monoids, the question was considered and some results obtained in \cite[Section 8]{m_3}. (The atomic structure of submonoids of $\Q_+$ was recently investigated in several papers of F.~Gotti, see for example \cite{g}.) 

\medskip
(4) The following is a natural {\bf question} about Diagram 1 and Dia\-gram 2: what would be the weakest type (WT) of integral domains such that all integrally closed domains and all AP domains are WT, and such that all the types of domains from Diagram 3 below (WT) are equivalent in the context of the monoid domains $F[X;M]$, where $M$ is a submonoid of $\Q_+$?

\medskip
(5) In \cite{cz} the notion of U-UFD domains is defined. It was shown that every AP domain is U-UFD and an example is constructed of an U-UFD domain which is not AP. A {\bf question} one can ask is if the notions of AP and U-UFD domains are distinct in the context of the monoid domains $F[X;M]$ with $M\subseteq \Q_+$, and if they are, to characterize the monoids $M\subseteq \Q_+$ for which $F[X;M]$ is U-UFD. One could then use this characterization to obtain other examples of U-UFD domains that are not AP.  

\medskip
(6) Another {\bf question} one can consider is to  characterize the monoids $M\subseteq \Q_+$ for which the monoid domains $F[X;M]$ are IDF, HFD, FFD, BFD (these domains and relations between them are analyzed in detail in \cite{aaz}).

\bigskip
\noindent {\bf Acknowledgment.}
The authors are thankful to Steve Seif for his comments on an earlier draft of the paper and, in particular, for pointing out that the essential generators of a monoid $(M,+)$ are precisely the atoms of $M$.

\bigskip\bigskip
\centerline{
\begin{tikzpicture}
\node[draw] (INT_CL) at (2.01,-5) {\begin{tabular}{c} Integrally\\ closed\end{tabular}};
\node[draw] (PRUF) at (2, -8.2) {Pr\"ufer};
\node[draw] (WT) at (4,-1.5) {WT};
\node[draw] (BEZ) at (6.01,-10.6) {B\'ezout};
\node[draw] (AP) at (5.985,-3) {AP};
\node[draw] (PRE_SCH) at (5.98,-5) {Pre-Schreier};
\node[draw] (SCH) at (5.99,-7) {Schreier};
\node[draw] (GCD) at (6,-8.6) {GCD};
\node[draw] (PC) at (8.5,-8.6) {PC};
\node[draw] (MIP) at (9.7,-10.6) {MIP};
\draw[-implies, double equal sign distance] (INT_CL) -- (WT);
\draw[-implies, double equal sign distance] (AP) -- (WT);
\draw[-implies, double equal sign distance] (PRUF) -- (INT_CL);
\draw[-implies, double equal sign distance] (BEZ) -- (PRUF);
\draw[-implies, double equal sign distance] (BEZ) -- (GCD);
\draw[-implies, double equal sign distance] (GCD) -- (SCH);
\draw[-implies, double equal sign distance] (SCH) -- (PRE_SCH);
\draw[-implies, double equal sign distance] (SCH) -- (INT_CL);
\draw[-implies, double equal sign distance] (PRE_SCH) -- (AP);
\draw[-implies, double equal sign distance] (BEZ) -- (PC);
\draw[-implies, double equal sign distance] (MIP) -- (PC);
\draw[-implies, double equal sign distance] (PC)  ..controls (8.5,-4).. (AP);
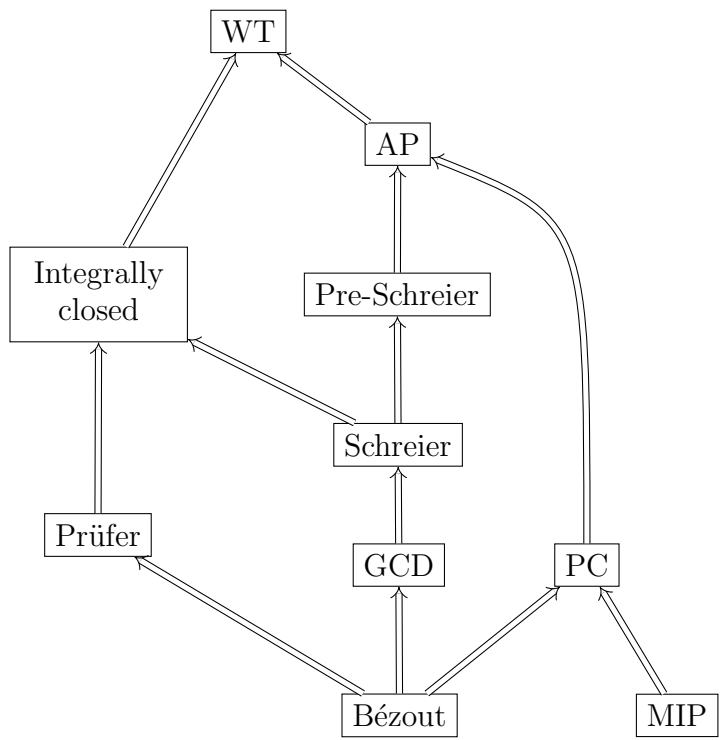
\node (title) at (5.2,-12.5) {\phantom{XXX}\begin{tabular}{c} Diagram 3: Some equivalent types of domains\\ in the $F[X;M]$ context with $M\subseteq \Q_+$\phantom{XXX}\end{tabular}};
\end{tikzpicture}
}

\end{document}